\documentclass[12pt]{article}
\usepackage{mycommands}
\usepackage{tikz}
\title{Interpolation for Curves in Projective Space with Bounded Error}

\begin{document}
\maketitle

\begin{abstract}    
Given $n$ general points $p_1, p_2, \ldots, p_n \in \pp^r$
it is natural to ask whether there is a curve of given
degree $d$ and genus $g$ passing through them; by counting dimensions
a natural conjecture is that such a curve exists if and only if
\[n \leq \left\lfloor \frac{(r + 1)d - (r - 3)(g - 1)}{r - 1}\right\rfloor.\]
The case of curves with \emph{nonspecial} hyperplane section was recently
studied in \cite{firstpaper}, where the above conjecture was shown to hold
with exactly three exceptions.

In this paper, we prove a ``bounded-error analog''
for \emph{special} linear series
on general curves; more precisely we show that existence of such a curve subject
to the stronger inequality
\[n \leq \left\lfloor \frac{(r + 1)d - (r - 3)(g - 1)}{r - 1}\right\rfloor - 3.\]
Note that the $-3$ cannot be replaced with $-2$ without introducing
exceptions
(as a canonical curve in $\pp^3$ can only pass through $9$ general points, while
a naive dimension count predicts~$12$).

We also use the same technique to prove that the twist of the normal bundle $N_C(-1)$
satisfies interpolation for curves whose degree is sufficiently large relative to their
genus, and deduce from this that the number of general points contained in
the hyperplane section of a general curve is at least
\[\min\left(d, \frac{(r - 1)^2 d - (r - 2)^2 g - (2r^2 - 5r + 12)}{(r - 2)^2}\right).\]

As explained in \cite{over}, these results play a key role in the author's proof
of the Maximal Rank Conjecture \cite{mrc}.
\end{abstract}

\section{Introduction}

If $C$ is a general curve, equipped with a general map $f \colon C \to \pp^r$ of degree~$d$,
it is natural to ask how many general points are contained in $f(C)$.
This problem has been studied in many cases, including for nonspecial curves \cite{firstpaper},
for space curves \cite{vogt}, and for canonical curves \cite{can}.
To state the problem
precisely, we make the following definition:

\begin{defi}
We say a stable map $f \colon C \to \pp^r$ of degree $d$ from a curve of genus $g$
is a \emph{Brill--Noether curve (BN-curve)} if it is a limit of nondegenerate degree~$d$ maps
$C' \to \pp^r$
with $[C'] \in \bar{M}_g$ of general moduli.

If $[f]$ lies in a unique component of $\bar{M}_g(\pp^r, d)$,
we say $f$ is an \emph{interior curve}.
\end{defi}

\noindent
The celebrated Brill--Noether theorem
then asserts that BN-curves exist if and only if
\[\rho(d, g, r) := (r + 1)d - rg - r(r + 1) \geq 0.\]
Moreover, for $\rho(d, g, r) \geq 0$, there is only one component
of $\bar{M}_g(\pp^r, d)$ (respectively $\bar{M}_{g, n}(\pp^r, d)$)
corresponding to BN-curves (respectively marked BN-curves); we write
$\bar{M}_g(\pp^r, d)^\circ$ (respectively $\bar{M}_{g, n}(\pp^r, d)^\circ$) for that component.
The question posed at the beginning then amounts to asking when
the natural map $\bar{M}_{g, n}(\pp^r, d)^\circ \to (\pp^r)^n$ is dominant.
In order for this to happen, it is evidently necessary for
\[(r + 1) d - (r - 3)(g - 1) + n = \dim \bar{M}_{g, n}(\pp^r, d)^\circ \geq \dim (\pp^r)^n = rn,\]
or equivalently,
\[n \leq \frac{(r + 1)d - (r - 3)(g - 1)}{r - 1}.\]
However, this is not sufficient: When $(d, g, r) = (6, 4, 3)$, the above equation gives $n \leq 12$;
but every canonical curve in $\pp^3$ lies on a quadric, and so can only pass through $9$ general points
(three less than expected).
Our main theorem implies that the above condition is ``as close as possible to sufficient
given the above example'' --- a bound which is (barely!) good enough
to prove the Maximal Rank Conjecture, as explained in \cite{over}:

\begin{thm} \label{main}
There exists a BN-curve of degree $d$ and genus $g$ in $\pp^r$
(with $\rho(d, g, r) \geq 0$), passing through $n$ general points, if
\[(r - 1) n \leq (r + 1)d - (r - 3)(g - 1) - 2r.\]
In particular, such a curve exists so long as
\[n \leq \frac{(r + 1)d - (r - 3)(g - 1)}{r - 1} - 3.\]
\end{thm}

\begin{rem}
There are examples of smooth curves passing through many more general points
than the formula given in Theorem~\ref{main}, given by complete intersections.
However such curves are not useful for studying either
the extrinsic geometry of general curves (e.g.\ in the proof of the Maximal Rank Conjecture c.f.\ \cite{over}),
nor for studying the intrinsic geometry of general curves and the geometry of $\bar{M}_g$
(e.g.\ in the construction of moving curves in $\bar{M}_g$ c.f.\ \cite{nasko}).
For such applications, it is essential to know that the smooth curves
found passing through the given general points are BN-curves,
which is exactly what is guaranteed by our theorem.
\end{rem}

Using similar techniques, we also study the analogous question for hyperplane sections:
For a hyperplane $H$, we ask how many general points in $H$
are contained in $f(C) \cap H$. 
This problem has been studied for $r \leq 4$ in \cite{quadrics},
but remains open in higher-dimensional projective spaces.
As with the case of general points in $f(C)$, it is evidently necessary for $n \leq d$ and
\[(r + 1)d - (r - 3)(g - 1) = \dim \bar{M}_g(\pp^r, d)^\circ \geq \dim H^n + \dim G = (r - 1)n + (r + 1),\]
where $G \subset \aut \pp^r$ denotes the subgroup
fixing $H$ pointwise;
or upon rearrangement,
\[n \leq \min\left(d, \frac{(r + 1)d - (r - 3)g - 4}{r - 1}\right).\]
Again this is not sufficient, for a similar reason:
When $(d, g, r) = (6, 4, 3)$, the above equation gives $n \leq 6$;
but every canonical curve in $\pp^3$ lies on a quadric, so its hyperplane section lies on a conic,
and can thus only pass through $5$ general points.
Our second main theorem implies the above condition is nevertheless also not far from sufficient:

\begin{thm} \label{main-1s}
The hyperplane section of a general BN-curve of degree $d$ and genus $g$ in $\pp^r$
contains $d - n$ general points (with $0 \leq n \leq d$) if
\[(2r - 3)(d + 1) - (r - 2)^2 (g - n) - 2r^2 + 3r - 9 \geq 0.\]
Or equivalently, the number of general points contained in the hyperplane section
of a general BN-curve of degree $d$ and genus $g$ in $\pp^r$ is at least
\[\min\left(d, \frac{(r - 1)^2 d - (r - 2)^2 g - (2r^2 - 5r + 12)}{(r - 2)^2}\right).\]
\end{thm}

These theorems are proven by studying the normal bundle of the general marked
BN-curve $f \colon (C, p_1, p_2, \ldots, p_n) \to \pp^r$. Namely, let $C$ be a nodal curve,
$p_1, p_2, \ldots, p_n$ be smooth points of $C$, and $f \colon C \to \pp^r$
be a map.
As long as $f$ is unramified,
basic deformation theory (which is a special case of more general results of \cite{bf} and \cite{b},
as explained in Section~2 of~\cite{ghs}) implies the map
\[f \mapsto (f(p_1), f(p_2), \ldots, f(p_n))\]
from the corresponding Kontsevich space to $(\pp^r)^n$ is smooth at $[f]$ if
\[H^1(N_f(-p_1-\cdots -p_n)) = 0.\]
Here, $N_f$ denotes the normal bundle of the map $f \colon C \to \pp^r$,
which is defined as
\[N_f = \ker(f^* \Omega_{\pp^r} \to \Omega_C)^\vee.\]

The conclusions of Theorems~\ref{main} and~\ref{main-1s}
are trivially always true when $r = 1$ (in which case $f$ is surjective);
we may therefore suppose $r \geq 2$, which implies
a general BN-curve is unramified.
Since a map between irreducible varieties is dominant if it is
generically smooth,
Theorem~\ref{main} then reduces to the assertion that
$H^1(N_f(-p_1-\cdots -p_n)) = 0$
for $f \colon (C, p_1, p_2, \ldots, p_n) \to \pp^r$ a general marked BN-curve. This condition is visibly
open, so to prove Theorem~\ref{main} it suffices to exhibit an unramified marked
BN-curve $f \colon (C, p_1, p_2, \ldots, p_n) \to \pp^r$, of each degree $d$ and genus $g$
satisfying the assumptions of Theorem~\ref{main},
for which
\begin{equation} \label{h10}
H^1(N_f(-p_1-\cdots -p_n)) = 0.
\end{equation}

In proving Theorem~\ref{main-1s}, we first consider
a related problem: We ask when a general BN-curve has
general hyperplane section, and in that case through
how many additional independently general points
it passes.
To study this question, we note that, as above,
so long as $f$ is unramified and transverse to a hyperplane $H$,
the map
\[f \mapsto (f(p_1), f(p_2), \ldots, f(p_n), f(C) \cap H)\]
from the corresponding Kontsevich space to $(\pp^r)^n \times \sym^d H$ is smooth at $[f]$ if
\begin{equation} \label{h11}
H^1(N_f(-1)(-p_1-\cdots -p_n)) = 0.
\end{equation}

Conditions \eqref{h10} and \eqref{h11} above
are closely related to the property of \emph{interpolation} for the normal bundle $N_f$
and its twist $N_f(-1)$:

\begin{defi}
We say that a vector bundle $\mathcal{E} \to C$ on a curve $C$ \emph{satisfies interpolation}
if, for a general effective divisor $D$ of any degree,
\[H^0(\mathcal{E}(-D)) = 0 \tor H^1(\mathcal{E}(-D)) = 0.\]
\end{defi}

Note that if $\mathcal{E} \to C$ satisfies interpolation,
then $H^1(\mathcal{E}(-p_1-\cdots-p_n)) = 0$ for general points $p_1, p_2, \ldots, p_n \in C$
if and only if $n \leq \chi(\mathcal{E}) / \rk(\mathcal{E})$.
The above argument therefore shows that:
\begin{itemize}
\item If $N_f$ satisfies interpolation for $f$ a general BN-curve of degree $d$
and genus $g$, then $f(C)$ can pass through $n$ general points if and only if
\[n \leq \frac{(r + 1) d - (r - 3)(g - 1)}{r - 1}.\]
\item If $N_f(-1)$ satisfies interpolation for $f$ a general BN-curve of degree $d$
and genus $g$, then $f(C)$ has a general hyperplane section, and passes through $n$
additional general points (independent of its hyperplane section) if and only if
\[n \leq \frac{2d - (r - 3)(g - 1)}{r - 1}.\]
\end{itemize}

We now state the following theorem on interpolation
for the twist of the normal bundle, from which Theorem~\ref{main-1s} will be deduced:

\begin{thm} \label{main-1}
If $f \colon C \to \pp^r$ is a general BN-curve of degree $d$ and genus $g$,
then $N_f(-1)$ satisfies interpolation provided that
\[(2r - 3)d - (r - 2)^2 g - 2r^2 + 3r - 9 \geq 0.\]
\end{thm}

Note that Theorem~\ref{main-1} is trivial when $r = 1$
(in which case $N_f = 0$).
For $r = 2$, Theorems~\ref{main}, \ref{main-1s}, and \ref{main-1}
follow immediately from the above discussion,
once we note that $N_f \simeq K_C(3)$ and $N_f(-1) \simeq K_C(2)$
are nonspecial line bundles so in particular satisfy interpolation.
More precise versions of Theorems~\ref{main}, \ref{main-1s}, and \ref{main-1}
are already known for $r = 3$ by work of Vogt \cite{vogt},
and for $r = 4$ by work of the author and Vogt \cite{p4}.
\emph{We may therefore assume for simplicity that $r \geq 5$
for the remainder of the paper.}
(Although we note that with a bit more care,
the techniques used here apply to lower
values of $r$ too; in particular, they cannot be used to prove
a sharper version of Theorem~\ref{main} with the $-3$ replaced by a $-2$,
as that would contradict the known counterexample with $r = 3$ mentioned above.)

The key idea to prove our main theorems is to degenerate $f$
to a map $f^\circ \colon C \cup_\Gamma D \to \pp^r$ from a reducible
curve, so that $f^\circ|_C$ and $f^\circ|_D$ are both nonspecial,
and so that $f^\circ|_D$ factors through a hyperplane $H$.
We then use a trick of \cite{quadrics} (which we recall
as Lemma~\ref{lm27} below) to reduce the desired statements
to facts about the normal bundles of $f^\circ|_C$ and $f^\circ|_D$,
which then follow from results of \cite{firstpaper} on interpolation
for nonspecial curves.

\paragraph{Note:} Throughout this paper, we work over an algebraically
closed field of characteristic zero.
Unless otherwise specified, all curves are assumed to be nodal.

\subsection*{Acknowledgements}

The author would like to thank Joe Harris for
his guidance throughout this research,
as well as other members of the Harvard and MIT mathematics departments
for helpful conversations.
The author would also like
to acknowledge the generous
support both of the Fannie and John Hertz Foundation,
and of the Department of Defense
(NDSEG fellowship).
Finally, the author would like to thank the anonymous referee
for various helpful comments on the manuscript.

\section{Proof of Theorem~\ref{main}}

In this section, we prove Theorem~\ref{main}.
Since Theorem~\ref{main}
holds when $d \geq g + r$ by Corollary~1.4 of~\cite{firstpaper},
we suppose that $d < g + r$ in this section.
The key input in our proof of Theorem~\ref{main} will be the following lemma from \cite{quadrics}:

\begin{lm}[Lemma~2.7 of~\cite{quadrics}] \label{lm27}
Let $f \colon C \cup_\Gamma D \to \pp^r$ be an unramified map from a nodal curve,
such that $f|_D$ factors as a composition of $f_D \colon D \to H$ with the inclusion of a hyperplane $\iota \colon H \subset \pp^r$,
while $f|_C$ is transverse to $H$ along $\Gamma$.

Let $E$ and $F$ be Cartier 
divisors supported on $C \smallsetminus \Gamma$ and $D \smallsetminus \Gamma$
respectively.
Suppose that, for some $i \in \{0, 1\}$,
\[H^i(N_{f_D}(-\Gamma-F)) = H^i(\oo_D(1)(\Gamma-F)) = H^i(N_{f|_C} (-E)) = 0.\]
Then we have
\[H^i(N_f(-E-F)) = 0.\]
\end{lm}

\begin{cor} \label{cor:needformain}
Let $f \colon C \cup_\Gamma D \to \pp^r$ be an interior BN-curve of the form appearing in
Lemma~\ref{lm27} (first paragraph), such that $f|_C$ and $f_D$ are BN-curves;
$f(\Gamma)$ is a set of $s + r$ general points in $H$,
where $s = g + r - d$;
and $D$ is of genus $(r - 2)t$ and $f|_D$ is of degree $(r - 2)t + r - 1$,
for some integer $t$.

If $s - 1 \leq 2t \leq r + s + 1$, and $d - (r - 2)t - 2r + 1 \neq 2$ if $r = 5$,
then Theorem~\ref{main} holds for curves of degree $d$ and genus $g$ in $\pp^r$.
\end{cor}
\begin{proof}
Note that our assumptions force $C$ to be of genus $g + 1 - (r - 2)t - s - r = d - (r - 2)t - 2r + 1$,
and $f|_C$ to be of degree $d - ((r - 2)t + r - 1) = d - (r - 2)t - r + 1$.
In particular, $f|_C$ is nonspecial.

Since $f$ is an interior curve, and $\Gamma$ is a general set of points,
we may deform $f$
to assume that $(f_D, \Gamma)$ is general in the component of $M_{(r - 2)t, s + r}(H, (r - 2)t + r - 1)$
corresponding to BN-curves,
and that $f|_C$ is general in the component of $M_{d - (r - 2)t - 2r + 1}(\pp^r, d - (r - 2)t - r + 1)$
corresponding to BN-curves.

By Lemma~\ref{lm27}, it suffices to produce Cartier divisors $E$ and $F$,
supported on $C \smallsetminus \Gamma$ and $D \smallsetminus \Gamma$
respectively, satisfying
\[\deg E + \deg F = \left\lfloor \frac{(r + 1)d - (r - 3)(g - 1)}{r - 1} - \frac{2r}{r - 1}\right\rfloor,\]
for which
\begin{equation} \label{n1}
H^1(N_{f_D}(-\Gamma-F)) = H^1(\oo_D(1)(\Gamma-F)) = H^1(N_{f|_C} (-E)) = 0.
\end{equation}

By Corollary~1.4 of~\cite{firstpaper} for $f_D$, and for $f|_C$
(except when $r = 5$ and
and $d - (r - 2)t - 2r + 1 = 2$ which does not hold by assumption),
when $E$ and $F$ are general divisors
the conditions \eqref{n1} reduce to:
\begin{align}
\deg F + s + r &\leq \frac{r \cdot ((r - 2)t + r - 1) - (r - 4)((r - 2)t - 1)}{r - 2} = r + 2 + 4t \label{i1} \\
\deg F - s - r &\leq ((r - 2)t + r - 1) + 1 - (r - 2)t = r \label{i2} \\
\begin{split}
\deg E &\leq \frac{(r + 1) (d - (r - 2)t - r + 1) - (r - 3)(d - (r - 2)t - 2r)}{r - 1} \label{i3}\\
&= \frac{4d - (4r - 8)t + r^2 - 6r + 1}{r - 1}.
\end{split}
\end{align}
We will make \eqref{i1} an equality by choosing $\deg F = 4t + 2 - s$;
upon rearrangement, \eqref{i2} becomes $2t \leq r + s + 1$, which holds by assumption.
Finally, \eqref{i3} becomes
\[\left\lfloor \frac{(r + 1)d - (r - 3)(g - 1)}{r - 1} - \frac{2r}{r - 1}\right\rfloor - (4t + 2 - s) \leq \frac{4d - (4r - 8)t + r^2 - 6r + 1}{r - 1};\]
this in turn follows from
\[\frac{(r + 1)d - (r - 3)(g - 1)}{r - 1} - \frac{2r}{r - 1} - (4t + 2 - s) \leq \frac{4d - (4r - 8)t + r^2 - 6r + 1}{r - 1},\]
or upon rearrangement, $2t \geq s - 1$, which also holds by assumption.
\end{proof}

Our goal is thus to construct curves satisfying the assumptions of Corollary~\ref{cor:needformain}.
For this purpose,
we will first need the following lemma:

\begin{lm} \label{lm:aut} Let $f \colon C \to \pp^r$ be an unramified map from a curve
with $H^1(N_f) = 0$.
If $\Gamma \subset C$ is a set of $n \leq r + 2$ points with $f(\Gamma)$
in linear general position, then $H^1(N_f(-\Gamma)) = 0$.
\end{lm}
\begin{proof}
The (long exact sequence in cohomology attached to the)
short exact sequence of sheaves
\[0 \to N_f(-\Gamma) \to N_f \to N_f|_\Gamma \to 0\]
reduces our problem to showing
$H^0(N_f) \to H^0(N_f|_\Gamma)$
is surjective. For this, we use the commutative diagram
\[\begin{CD}
H^0(T_{\pp^r}) @>>> H^0(T_{\pp^r}|_{f(\Gamma)}) \\
@VVV @VVV  \\
H^0(N_f) @>>> H^0(N_f|_\Gamma).
\end{CD}\]
The top horizontal map is surjective since $n \leq r + 2$ and $f(\Gamma)$ is in linear general position
by assumption, and the right vertical map is always surjective.
Consequently, the bottom horizontal map is surjective as desired.
\end{proof}

With this out of the way, the construction of $f$
can be done in most cases by the following lemma:

\begin{lm} \label{lm:donemain} There exists a reducible interior BN-curve $f\colon C \cup_\Gamma D \to \pp^r$
of the form appearing in Corollary~\ref{cor:needformain} (first paragraph),
with $t = \lfloor s/2 \rfloor$.
\end{lm}
\begin{proof}
We argue by induction on $d$, for a stronger hypothesis: That such a curve $f$
exists which, in addition, satisfies:
\begin{itemize}
\item $f(C)$ passes through $2$ (if $s$ is odd) or $1$ (if $s$ is even) points in $\pp^r$ that are
general, independently from $f(\Gamma)$;
\item $f(D)$ passes through a point in $H$ that is
general, independently from $f(\Gamma)$ and the above general point in $f(C)$, provided that $s$ is even;
\item and $H^1(N_f) = 0$.
\end{itemize}
First we consider the case $\rho(d, g, r) = 0$, which implies $d = r(s + 1)$ and $g = (r + 1)s$.

When $s = 1$, we take $f|_C$ to be a general elliptic normal curve --- which
has a general hyperplane section, and passes through $2$ additional
independently general points in $\pp^r$ as required, by Lemma~6.1 of~\cite{rbn}.
We let $f_D$ be a rational normal curve in $H$
passing through all points of intersection of $f|_C(C)$ with $H$. The union is a BN-curve by Theorem~1.7 of~\cite{rbn},
which is an interior curve satisfying $H^1(N_f) = 0$
by combining Lemmas~3.2, 3.3, and~3.4 of~\cite{rbn}.

For the inductive argument, we assume the given statement for $s - 1$
and seek to verify it for $s$.
Let $f_0 \colon C_0 \cup_{\Gamma_0} D_0 \to \pp^r$ be such a curve of degree
$d_0 = rs$ and genus $g_0 = (r + 1)(s - 1)$.

\paragraph{\boldmath If $s$ is even:} We pick general subsets of $3$ points $\Delta_C \subset C_0$
and of $r - 1$ points $\Delta_D \subset D_0$.
Write $\Lambda_C \simeq \pp^2$ and
$\Lambda_D \simeq \pp^{r - 2}$ for the linear spans of $f_0(\Delta_C)$ and $f_0(\Delta_D)$
respectively.
Note that a line passing through two points of $f_0(\Delta_C)$
is general (independent of $f_0(\Gamma_0)$)
by our inductive hypothesis;
in particular, $\Lambda_C \cap H$ contains a point which is general in $H$
(independent of $f_0(\Gamma_0)$),
and $\Lambda_C$ contains an additional independently general point in $\pp^r$.

Since $\Lambda_D \subset H$ is a
general hyperplane section, $p = \Lambda_C \cap \Lambda_D$
is a general point of $\Lambda_C \cap H$,
and is thus general in $H$ (independent of $f_0(\Gamma_0)$);
moreover if $q_1 \in \Lambda_C$ and $q_2 \in \Lambda_D$ are general,
then $q_1$ and $q_2$ are general in $\pp^r$ and $H$ respectively
(independent of $f_0(\Gamma_0)$ and $p$).
Let $C' \subset \Lambda_C$ be a rational normal curve (i.e.\ of degree $\dim \Lambda_C = 2$)
through $f_0(\Delta_C) \cup \{p, q_1\}$, and $D' \subset \Lambda_D$ be a rational normal curve
through $f_0(\Delta_D) \cup \{p, q_2\}$.
Then we will show that
\[f \colon (C_0 \cup_{\Delta_C} C') \cup_{\Gamma_0 \cup \{p\}} (D_0 \cup_{\Delta_D} D') \to \pp^r\]
gives the required curve.

Writing $f$ as $(C_0 \cup_{\Gamma_0} D_0) \cup_{\Delta_C \cup \Delta_D} (C' \cup_p D') \to \pp^r$,
iteratively applying Theorem~1.6 of \cite{firstpaper} shows it is a BN-curve.
In addition, applying Theorem~1.6 of \cite{firstpaper} shows
$f|_{C_0 \cup_{\Delta_C} C'}$ is a BN-curve to $\pp^r$,
and $f|_{D_0 \cup_{\Delta_D} D'}$ is a BN-curve to $H$,
as desired.

Moreover,
Lemmas~3.2, 3.3 and~3.4 of~\cite{firstpaper} imply
$H^1(N_{f|_{C' \cup_p D'}}) = 0$. Applying Lemma~\ref{lm:aut},
we conclude $H^1(N_{f|_{C' \cup_p D'}}(-\Delta_C - \Delta_D)) = 0$.
Together with
our inductive hypothesis,
using Lemmas~3.3 and~3.4 of \cite{firstpaper}, this implies
$H^1(N_f) = 0$ as desired.

Finally, we note that $f(C_0 \cup_{\Delta_C} C')$ passes through $1$ point in $\pp^r$ that is general
independent from $f(\Gamma_0 \cup \{p\})$, namely $q_1$;
and $f(D_0 \cup_{\Delta_D} D')$ passes through $1$ point in $H$
that is general
independent from $f(\Gamma_0 \cup \{p\}) \cup \{q_1\}$, namely $q_2$.

\paragraph{\boldmath If $s$ is odd:}
We pick a general subset $\Delta \subset C_0$ of $r + 1$ points,
a general point $q_1 \in f(C_0)$,
a general point $p \in D_0$, and a general point $q_2 \in \pp^r$.
By our inductive hypothesis, $q_1 \in \pp^r$ is general independent from $f_0(\Gamma_0)$,
and $f_0(p) \in H$ is general independent from $f_0(\Gamma_0)$.

Let $C'$ be a rational normal curve through $f(\Delta) \cup \{p, q_2\}$.
Then we will show that
\[f \colon (C_0 \cup_\Delta C') \cup_{\Gamma \cup \{p\}} D_0 \to \pp^r\]
gives the required curve.

Writing $f$ as $(C_0 \cup_{\Gamma_0} D_0) \cup_{\Delta \cup \{p\}} C' \to \pp^r$,
applying Theorem~1.6 of \cite{firstpaper} shows it is a BN-curve.
In addition, applying Theorem~1.6 of \cite{firstpaper} shows
$f|_{C_0 \cup_\Delta C'}$ is a BN-curve to $\pp^r$,
as desired.

Moreover,
Lemma~3.2 of~\cite{firstpaper} implies
$H^1(N_{f|_{C'}}(-\Delta-p)) = 0$. 
Together with
our inductive hypothesis,
using Lemmas~3.3 and~3.4 of \cite{firstpaper}, this implies
$H^1(N_f) = 0$ as desired.

Finally, we note that $f(C_0 \cup_\Delta C')$ passes through $2$ points in $\pp^r$ that are general
independent from $f(\Gamma_0 \cup \{p\})$, namely $\{q_1, q_2\}$.
\end{proof}

\begin{proof}[Proof of Theorem~\ref{main}]
Combining Lemma~\ref{lm:donemain} with Corollary~\ref{cor:needformain}
proves Theorem~\ref{main} unless we have
$r = 5$ and $d - (r - 2) \lfloor s/2 \rfloor - 2r + 1 = 2$;
or upon rearrangement $r = 5$ and $d = 3 \cdot \lfloor s/2 \rfloor + 11$.
Note that, in these cases,
\[s = g + 5 - d = \frac{d - 5 - \rho(d, g, 5)}{5} \leq \frac{d - 5}{5};\]
consequently
\[d \leq 3 \cdot \frac{d - 5}{10} + 11 \imp d \leq \frac{95}{7} < 14,\]
and so
\[s \leq \frac{d - 5}{5} < \frac{9}{5} < 2 \imp s = 1,\]
which gives $d = 11$ and thus $g = 7$.

It thus remains to prove Theorem~\ref{main} in the case $(d, g, r) = (11, 7, 5)$,
i.e.\ to prove that such a curve
can pass through $11$ general points.
But a curve of degree $10$ and genus $6$ can pass through $11$ general points
by work of Stevens \cite{can}, and the union of a curve of degree $10$ and genus $6$
with a $2$-secant line gives a curve of the required degree and genus, which is a BN-curve
by Theorem~1.6 of~\cite{rbn}.
\end{proof}

\section{The Twist}

We now turn to studying interpolation for the twist $N_C(-1)$.
In greater generality, we make the following definition:

\begin{defi} \label{def:good}
Let $d, g, r, n$ be nonnegative integers with $n \leq d$ and $\rho(d, g, r) \geq 0$;
take $f \colon C \to \pp^r$ to be a general BN-curve of degree $d$ and genus $g$.

We say $(d, g, r, n)$ is \emph{good} if the general hyperplane section $f(C) \cap H$
contains $d - n$ general points in $H$.

We say $(d, g, r, n)$ is \emph{excellent} if $N_f(-D)$ satisfies interpolation,
where $D \subset C$ is a divisor of degree $d - n$ supported in a general hyperplane section.
\end{defi}

By definition, the twist $N_f(-1)$ satisfies interpolation
if and only if $(d, g, r, 0)$ is excellent.
Mirroring our previous argument, we will begin from knowledge that
some range of degrees and genera are excellent:

\begin{prop} \label{base:excellent}
Let $d, g, r, n$ be nonnegative integers with
$n \leq d$ and $\rho(d, g, r) \geq 0$.
Then $(d, g, r, n)$ is excellent provided that:
\begin{align*}
d &\geq g + r, \\
(d, g, r) &\notin \{(5, 2, 3), (6, 2, 4), (7, 2, 5)\}, \quad \text{and} \\
2d + (r - 1) n &\geq (2r - 4)g - r + 3.
\end{align*}
\end{prop}
\begin{proof}
This follows from combining
Theorem~1.3 and Proposition~4.12 of~\cite{firstpaper}.
\end{proof}

Our goal in this section is to prove Theorem~\ref{main-1}, by showing
that $N_f(-1)$ satisfies interpolation, subject to the inequality
\begin{equation} \label{assume}
(2r - 3)d - (r - 2)^2 g - 2r^2 + 3r - 9 \geq 0.
\end{equation}

\begin{lm} Theorem~\ref{main-1} holds for curves of degree $d$ and genus $g$ in $\pp^r$
unless
\begin{equation} \label{notdonealready}
2d \leq (2r - 4) g - r + 2
\end{equation}
and
\begin{equation} \label{gmin}
g \geq \begin{cases}
5 & \text{if $r \in \{5, 6\}$} \\
4 & \text{otherwise.}
\end{cases}.
\end{equation}
\end{lm}
\begin{proof}
By our assumption that $r \geq 5$, note that \eqref{assume} implies $d \geq g + r$;
in addition, \eqref{assume} is not satisfied for $(d, g, r) = (7, 2, 5)$.
In particular, if $2d \geq (2r - 4) g - r + 3$, then Proposition~\ref{base:excellent}
implies that $N_f(-1)$ satisfies interpolation as desired.
Note that, using \eqref{assume}, this implies
\[g \geq 3 + \frac{5r + 12}{2r^2 - 6r + 4},\]
which yields \eqref{gmin}.
\end{proof}

\begin{cor} \label{cor:needformain-1}
Let $f \colon C \cup_\Gamma D \to \pp^r$ be an interior BN-curve
of the form appearing in Lemma~\ref{lm27} (first paragraph),
with $f|_C$ and $f_D$ BN-curves too,
so that $D$ is of genus $t \geq 1$, and is of degree $(r - 3)t + 1$,
and $\Gamma$ is a set of $t$ general points in $H$. If
\begin{align}
t &\geq 2, \label{d1} \\
(r, t) &\notin \{(5, 2), (6, 2)\}, \label{d2} \\
g &\geq 2t - 1, \label{c1} \\
d &\geq g + r + 2 + (r - 5)t, \label{c2} \\
(d - (r - 3)t, g - 2t, r) &\neq (8, 1, 5), \quad \text{and} \label{c3}\\
2d - (2r - 4)g + (2r - 2)t &\geq r + 1, \label{c4}
\end{align}
then Theorem~\ref{main-1} holds for curves of degree $d$ and genus $g$ in $\pp^r$.
\end{cor}
\begin{proof}
As in the proof of Corollary~\ref{cor:needformain},
we may deform $f$
to assume that $(f_D, \Gamma)$ is general in the component of $M_{t, t}(H, (r - 3)t + 1)$
corresponding to BN-curves,
and that $f|_C$ is general in the component of $M_{g - 2t + 1}(\pp^r, d - (r - 3)t - 1)$
corresponding to BN-curves.

Let $F = \oo_D(1)(p)$, and $E = \oo_C(1)(\Delta)$ for a general divisor $\Delta \subset C$.
Note with these choices that
\[\oo_D(1)(\Gamma-F) = \oo_D(\Gamma) (-p) \tand N_{f_D}(-\Gamma-F) = N_{f_D}(-1)(-\Gamma - p)\]
both have Euler characteristic zero.
Moreover, $N_{f_D}(-1)$ satisfies interpolation by Proposition~\ref{base:excellent},
provided that
\begin{align*}
(r - 3)t + 1 &\geq t + r - 1 \\
((r - 3)t + 1, t, r - 1) &\notin \{(5, 2, 3), (6, 2, 4), (7, 2, 5)\} \\
2((r - 3)t + 1) &\geq (2r - 6) t - r + 4.
\end{align*}
The last of these conditions is immediate for $r \geq 5$, while the first follows from \eqref{d1} and \eqref{d2},
and the second follows from \eqref{d2}.
We conclude that
\[H^0(\oo_D(1)(\Gamma-F)) = H^1(\oo_D(1)(\Gamma-F)) = H^0(N_{f_D}(-\Gamma-F)) = H^1(N_{f_D}(-\Gamma-F)) = 0.\]
In particular, applying Lemma~\ref{lm27},
we see that $N_f(-1)$ satisfies interpolation provided that
$N_{f|_C}(-1)$ does, which in turn (by Proposition~\ref{base:excellent}) follows from:
\begin{align*}
g - 2t + 1 &\geq 0 \\
d - (r - 3)t - 1 &\geq g - 2t + 1 + r, \\
(d - (r - 3)t - 1, g - 2t + 1, r) &\notin \{(5, 2, 3), (6, 2, 4), (7, 2, 5)\}, \quad \text{and} \\
2(d - (r - 3)t - 1) &\geq (2r - 4)(g - 2t + 1) - r + 3;
\end{align*}
or upon rearrangement, and using $r \geq 5$, the inequalities \eqref{c1}--\eqref{c4}.
\end{proof}

\begin{lm} \label{lm:formain-1-done}
If $d \geq (r - 2)t + 1$ in addition to the inequalities \eqref{d1}--\eqref{c4} are satisfied,
there exists a reducible interior BN-curve $f \colon C \cup_\Gamma D \to \pp^r$
of the form appearing in Corollary~\ref{cor:needformain-1}.
\end{lm}
\begin{proof}
We argue by induction on $t$
for a stronger hypothesis: That such a curve $f$
exists which, in addition, satisfies $H^1(N_f) = 0$,
and for which $f|_D$ passes through $2$ additional points in $H$
which are general independent of $\Gamma$.

Note first that \eqref{d1} and \eqref{d2}, in conjunction with Theorem~1.3 of \cite{firstpaper}, imply
there is a nonspecial curve $f_D \colon D \to H$ of degree $(r - 3)t + 1$ and genus $t$,
for which $N_{f_D}$ satisfies interpolation; as $r[(r - 3)t + 1] - (r - 4)(t - 1) \geq (r - 2) (t + 2)$,
the general such curve passes through at least $t + 2$ general points.

If $t \leq r + 2$, we let $f_C \colon C \to \pp^r$ be a general BN-curve of degree $d - (r - 3)t - 1$
and genus $g - 2t + 1$.
By \eqref{c1}--\eqref{c4} and Proposition~\ref{base:excellent}, we see that $f_C(C) \cap H$ consists of $d - (r - 3)t - 1$
general points; by our assumption that
$d \geq (r - 2) t + 1$, we have $d - (r - 3)t - 1 \geq t$.
From the previous paragraph,
there is a nonspecial curve $f_D \colon D \to H$ of degree $(r - 3)t + 1$ and genus $t$,
passing through a subset $\Gamma$ of $t$ points of $f_C(C) \cap H$,
and passing through $2$ additional general points in $H$.
Gluing $f_C$ to $f_D$ along $\Gamma$,
there exists such a curve $f \colon C \cup_\Gamma D \to \pp^r$, with $f|_C$ and $f_D$ general
BN-curves, meeting at a general set of $t$ points $\Gamma \subset H$.
The curve $f$ is a BN-curve by Theorem~1.9 of~\cite{rbn},
and is an interior curve with $H^1(N_f) = 0$
by combining Lemmas~3.2, 3.3, and~3.4 of~\cite{rbn}; by construction,
$f|_D$ passes through $2$ additional points in $H$
which are general independent of $\Gamma$.

For the inductive step, we suppose $t \geq r + 3$,
and let $(d_0, g_0, t_0) = (d - r + 3, g - 2, t - 1)$;
note that \eqref{d1} and \eqref{d2} are satisfied for $(d_0, g_0, t_0)$
by our assumption that $t \geq r + 3$,
and that \eqref{c1}--\eqref{c4} and $d_0 \geq (r - 2)t_0 + 1$ are immediate.
We may therefore let $f_0 \colon C_0 \cup_{\Gamma_0} D_0 \to \pp^r$
be such a reducible BN-curve of degree $d_0$ and genus $g_0$
with $\# \Gamma_0 = t_0$, which satisfies $H^1(N_{f_0}) = 0$,
and such that $f_0(D_0)$ passes through two additional points
that are general independent of $\Gamma_0$.
Let $\{q_1, q_2\}$ and $\{q_1', q_2'\}$ be two distinct such sets of points,
so both are general independent of $\Gamma_0$ (although of course not
necessarily independent from $\Gamma_0$ and each other).

Note that $\deg f_0|_{C_0} = (d - r + 3) - (r - 3)(t - 1) - 1 \geq t$ by assumption;
we may therefore pick a point $p \in (f_0(C_0) \cap H) \smallsetminus \Gamma_0$.
By \eqref{c1}--\eqref{c4} for $(d_0, g_0, t_0)$ established in the previous paragraph,
in conjunction with Proposition~\ref{base:excellent},
the hyperplane section $f_0(C_0) \cap H$ is general;
thus $p$ is general, independent of $\Gamma_0$,
and thus independent of $\Gamma_0 \cup \{q_1, q_2\}$.
Pick a linear subspace $\Lambda \simeq \pp^{r - 3} \subset H$ passing through $\{p, q_1, q_2\}$,
and let $D' \subset \Lambda$ be a rational normal curve through $\{p, q_1, q_2\}$.
We then claim
\[f \colon C_0 \cup_{\Gamma_0 \cup \{p\}} (D_0 \cup_{\{q_1, q_2\}} D') \to \pp^r\]
gives the required curve.

Writing $f$ as $(C_0 \cup_{\Gamma_0} D_0) \cup_{\{p, q_1, q_2\}} D' \to \pp^r$,
applying Theorem~1.6 of~\cite{rbn} shows it is a BN-curve.

Moreover, Lemmas~3.2 of \cite{rbn} implies $H^1(N_{f|_{D'}}(-p - q_1 - q_2)) = 0$.
Together with our inductive hypothesis, using Lemmas~3.3 and~3.4 of \cite{rbn},
this implies
$H^1(N_f) = 0$ as desired.

Finally, we note that $f(D_0 \cup_{\{q_1, q_2\}} D')$ passes through
two general points in $H$, independent of $\Gamma_0 \cup \{p\}$, namely $\{q_1', q_2'\}$.
\end{proof}

\begin{proof}[Proof of Theorem~\ref{main-1}]
Combining Corollary~\ref{cor:needformain-1} with Lemma~\ref{lm:formain-1-done},
all that remains to prove Theorem~\ref{main-1} is  
to solve a purely combinatorial problem: We must show that, subject to $r \geq 5$ and
\eqref{assume}--\eqref{gmin}, there exists an integer $t$ satisfying
$d \geq (r - 2)t + 1$ and \eqref{d1}--\eqref{c4}.
For this, we shall take 
\[t = \begin{cases}
s + 1 &\text{if $(r, s) \in \{(5, 2), (6, 2)\}$;} \\
s &\text{otherwise} \end{cases}
\twhere
s = \left\lceil\frac{(2r - 4)g - 2d + r + 1}{2r - 2} \right \rceil.\]
This automatically satisfies \eqref{d2}. By construction,
\[t \geq s \geq \frac{(2r - 4)g - 2d + r + 1}{2r - 2},\]
which implies \eqref{c4}.
Moreover, rearranging \eqref{notdonealready}, we obtain
\begin{equation} \label{s2}
\frac{(2r - 4)g - 2d + r + 1}{2r - 2} \geq \frac{2r - 1}{2r - 2} > 1 \imp s \geq 2,
\end{equation}
which, since $t \geq s$, implies \eqref{d1}.
If $(d, g, r) = ((r - 3)t + 8, 2t + 1, 5)$, then
\eqref{assume} becomes upon rearrangement $t \leq \frac{3}{4}$,
in contradiction to \eqref{d1}, which was just established;
this establishes \eqref{c3}.
In addition,
\begin{equation} \label{t-duh-bound}
s \leq \frac{(2r - 4)g - 2d + r + 1}{2r - 2} + 1 - \frac{1}{2r - 2}.
\end{equation}
Combined with \eqref{assume}, we obtain
\begin{align*}
s &\leq \frac{(2r - 4)g - 2d + r + 1}{2r - 2} + 1 - \frac{1}{2r - 2} + \frac{(2r - 3)d - (r - 2)^2 g - 2r^2 + 3r - 9}{(r - 2)(r - 1)} \\
&= \frac{d - 1}{r - 2} - \frac{r^2 + 16}{2 r^2 - 6 r + 4} \\
&\leq \frac{d - 1}{r - 2} - \begin{cases}
1 & \text{if $r \in \{5, 6\}$;} \\
0 & \text{otherwise.}
\end{cases}
\end{align*}
This implies $d \geq (r - 2)t + 1$, as desired.

For the \eqref{c1} and \eqref{c2}, the obvious upper bound \eqref{t-duh-bound}
will not suffice; instead we rearrange \eqref{assume} to produce
\[\frac{(2r - 4)g - 2d + r + 1}{2r - 2} \leq \frac{g - 1}{2} - \frac{(r - 1)g + 18}{4r^2 - 10r + 6},\]
which in turn implies $s \leq g/2$.
If $r \in \{5, 6\}$, then \eqref{gmin} gives $g \geq 5$.
Note that when $g = 5$, our bound $s \leq 5/2$ immediately gives $s \leq 4/2 = 2$.
Moreover, if $r \in \{5, 6\}$ and $g \geq 6$,
then 
$\frac{(r - 1)g + 18}{4r^2 - 10r + 6} > \frac{1}{2}$. We conclude that
\begin{equation} \label{t-bigr-bound}
s \leq \begin{cases}
\frac{g - 1}{2} & \text{if $r \in \{5, 6\}$;} \\
\frac{g}{2} & \text{otherwise.}
\end{cases}
\end{equation}
This bound implies \eqref{c1}. Moreover, for all $g \geq 8$, we have
\[\frac{(r - 2)^2 g + 2r^2 - 3r + 9}{2r - 3} \geq g + r + 2 + (r - 5) \cdot \begin{cases}
\frac{g + 1}{2} & \text{if $r \in \{5, 6\}$;} \\
\frac{g}{2} & \text{otherwise.}
\end{cases}\]
Combined with \eqref{assume} and \eqref{t-bigr-bound}, this yields \eqref{c2}
for $g \geq 8$.

From \eqref{gmin}, in order to complete the proof,
all that remains is to verify \eqref{c2} for $g \in \{4, 5, 6, 7\}$.
In these cases, \eqref{assume} becomes
\[d \geq \begin{cases}
3r - 5 + \frac{10}{2r - 3} & \text{if $g = 4$;} \\
\frac{7r - 13}{2} + \frac{r + 19}{4r - 6} & \text{if $g = 5$;} \\
4r - 8 + \frac{r + 9}{2r - 3} & \text{if $g = 6$;} \\
\frac{9 r - 18}{2} + \frac{r + 20}{4r - 6} & \text{if $g = 7$.}
\end{cases} \imp
d \geq \begin{cases}
3r - 4 & \text{if $g = 4$;} \\
\frac{7r - 12}{2} & \text{if $g = 5$;} \\
4r - 7 & \text{if $g = 6$;} \\
\frac{9 r - 17}{2} & \text{if $g = 7$.}
\end{cases}
\]
And \eqref{t-bigr-bound} becomes
\begin{equation} \label{s47}
s \leq \begin{cases}
2 & \text{if $g \in \{4, 5\}$;} \\
3 & \text{if $g \in \{6, 7\}$.}
\end{cases}
\end{equation}
This implies \eqref{c2}, except for the cases
$(r, g) \in \{(5, 5), (6, 4), (6, 5), (6, 6), (6, 7)\}$.
In those cases, \eqref{assume} becomes
\[d \geq \begin{cases}
\frac{89}{7} & \text{if $(r, g) = (5, 5)$;} \\
\frac{127}{9} & \text{if $(r, g) = (6, 4)$;} \\
\frac{143}{9} &\text{if $(r, g) = (6, 5)$;} \\
\frac{53}{3} & \text{if $(r, g) = (6, 6)$;} \\
\frac{175}{9} & \text{if $(r, g) = (6, 7)$.}
\end{cases} \imp
d \geq \begin{cases}
13 & \text{if $(r, g) = (5, 5)$;} \\
15 & \text{if $(r, g) = (6, 4)$;} \\
16 & \text{if $(r, g) = (6, 5)$;} \\
18 & \text{if $(r, g) = (6, 6)$;} \\
20 & \text{if $(r, g) = (6, 7)$.}
\end{cases}\]
Together with \eqref{s47} this implies the \eqref{c2} in these cases.
\end{proof}

\section{General Points in a Hyperplane Section \label{ghs}}

In this section, we investigate the number of general points contained
in the hyperplane section of a general BN-curve.
For the remainder of this section, we let $(d, g, r, n)$ denote nonnegative integers with
$\rho(d, g, r) \geq 0$ and $n \leq d$ and $r \geq 5$; our goal is to prove Theorem~\ref{main-1s},
which asserts that $(d, g, r, n)$ is good (c.f.\ Definition~\ref{def:good})
provided that
\[(2r - 3)(d + 1) - (r - 2)^2 (g - n) - 2r^2 + 3r - 9 \geq 0.\]

Our argument will be via induction, using the results of the preceding
section as a base case.
The various inductive arguments we shall use are as follows:

\begin{lm} \label{addcan}
Let $(d, g, r, n)$ be nonnegative integers with $\rho(d, g, r) \geq 0$
and $n \leq d - 4$.
Suppose that $g \geq 2r$ with strict inequality when $r = 5$,
and that $n \geq 2r - 6$. Then
$(d, g, r, n)$ is good if
$(d - 2r + 2, g', r, n - 2r + 6)$ is good, where
\[g' = \begin{cases}
g - 2r & \text{if $r \geq 6$;} \\
g - 11 & \text{if $r = 5$.}
\end{cases}\]
\end{lm}
\begin{proof}
Note that our assumptions imply $g' \geq 0$, and that $n' := n - 2r + 6$
and $d' := d - 2r + 2$ satisfy $0 \leq n' \leq d'$. Moreover,
\[\rho(d', g', r) \geq (r + 1)(d - 2r + 2) - r(g - 2r) - r(r + 1)
= (r + 1)d - rg - r(r + 1) + 2
\geq 0.\]
In particular, $(d', g', r, n')$ are nonnegative integers with $0 \leq n' \leq d'$
and $\rho(d', g', r) \geq 0$.
By assumption, $(d', g', r, n')$ is good.

So let $f_1 \colon C \to \pp^r$ be a general BN-curve
of degree $d'$ and genus $g'$,
whose hyperplane section $f_1(C) \cap H$
contains a set $S$ of $d' - n' = d - n - 4$ general points.
Pick a set $T$ of
$4$ independently general points in $H$, and let $H'$ be a general hyperplane
containing $T$ (in particular $H'$ is independently general from $C$,
and from $H'$ since $r \geq 5$).

Since $\rho(d, g, r) \geq 0$, we have
\[d \geq \frac{rg + r(r + 1)}{r + 1} \geq \frac{r \cdot 2r + r(r + 1)}{r + 1} > 3r - 2 \imp d \geq 3r - 1 \imp d' \geq r + 1.\]
Similarly, when $r = 5$, we have 
\[d \geq \frac{5g + 30}{6} \geq \frac{5 \cdot 11 + 30}{6} > 14 \imp d \geq 15 \imp d' \geq 7.\]
Putting these together, we conclude $d' \geq c$, where we define
\[c := \begin{cases}
r + 1 & \text{if $r > 5$;} \\
7 & \text{if $r = 5$.}
\end{cases}\]

By Lemma~6.1 of~\cite{rbn}, the hyperplane section $f_1(C) \cap H'$ contains
a set $\Gamma$ of $c$ general points.
By \cite{can}, there is a canonical curve $f_2 \colon D \to H'$
(of genus $r$) passing through $\Gamma$. We may then construct
$(f_1 \cup f_2) \colon C \cup_{\Gamma} D \to \pp^r$,
which is a BN-curve, by Theorem~1.9 of~\cite{rbn},
and is of degree $d$ and genus $g$
passing through the set $S \cup T \subset H$ of $d - n$ general points as desired.
\end{proof}

\begin{lm} \label{addhalf}
Write
\[a = \left\lceil \frac{r - 2}{2} \right \rceil,\]
and suppose $g \geq a + 1$ and $n \geq a$.
Then $(d, g, r, n)$ is good if
$(d - a, g - a - 1, r, n - a)$ is good.
\end{lm}
\begin{proof}
Note that our assumptions imply $g' := g - a - 1 \geq 0$, and that $n' := n - a$
and $d' := d - a$ satisfy $0 \leq n' \leq d'$. Moreover,
\[\rho(d', g', r) = (r + 1)(d - a) - r(g - a - 1) - r(r + 1)
= (r + 1)d - rg - r(r + 1) + r - a
\geq r - a \geq 0.\]
In particular, $(d', g', r, n')$ are nonnegative integers with $0 \leq n' \leq d'$
and $\rho(d', g', r) \geq 0$.
By assumption, $(d', g', r, n')$ is good.

So let $f \colon C \to \pp^r$ be a general BN-curve
of degree $d'$ and genus $g'$,
whose hyperplane section $f(C) \cap H$
contains $d' - n' = d - n$ general points.

By Theorem~1.8 of~\cite{rbn}, there exists a BN-curve
$\hat{f} \colon C \cup_\Gamma \pp^1 \to \pp^r$
with $\# \Gamma = a + 2$ and $\hat{f}|_{\pp^1}$ of degree $a$,
such that $\hat{f}|_C = f$.
In particular, $\hat{f}$ is a BN-curve of degree $d' + a = d$
and genus $g' + a + 1 = g$
whose hyperplane section contains the hyperplane section of $f$,
and thus contains $d - n$ general points as desired.
\end{proof}

\begin{lm} \label{addline}
Suppose that $g \geq 1$, and $n \geq 1$, and $\rho(d, g, r) \geq 1$.
Then $(d, g, r, n)$ is good if
$(d - 1, g - 1, r, n - 1)$ is good.
\end{lm}
\begin{proof}
Note that our assumptions imply $g' := g - 1 \geq 0$, and that $n' := n - 1$
and $d' := d - 1$ satisfy $0 \leq n' \leq d'$. Moreover,
\[\rho(d', g', r) = (r + 1)(d - 1) - r(g - 1) - r(r + 1)
= (r + 1)d - rg - r(r + 1) - 1
\geq 0.\]
In particular, $(d', g', r, n')$ are nonnegative integers with $0 \leq n' \leq d'$
and $\rho(d', g', r) \geq 0$.
By assumption, $(d', g', r, n')$ is good.

So let $f \colon C \to \pp^r$ be a general BN-curve
of degree $d'$ and genus $g'$,
whose hyperplane section $f(C) \cap H$
contains $d' - n' = d - n$ general points.

Pick $\{p, q\} \subset C$ general.
By Theorem~1.6 of~\cite{rbn}, the curve
$\hat{f} \colon C \cup_{\{p, q\}} \pp^1 \to \pp^r$,
where $\hat{f}|_{\pp^1}$ is a line,
is a BN-curve.
It is evidently of degree $d' + 1 = d$
and genus $g' + 1 = g$,
and its hyperplane section contains the hyperplane section of $f$,
and thus contains $d - n$ general points as desired.
\end{proof}

\begin{lm} \label{addline-1}
If $(d, g, r, n)$ is good, then so is $(d + 1, g, r, n)$.
\end{lm}
\begin{proof}
Let $f \colon C \to \pp^r$ be a general BN-curve
of degree $d$ and genus $g$,
whose hyperplane section $f(C) \cap H$
contains $d - n$ general points.

Pick a general point $p \in C$.
By Theorem~1.6 of~\cite{rbn}, the curve
$\hat{f} \colon C \cup_{\{p\}} \pp^1 \to \pp^r$,
where $\hat{f}|_{\pp^1}$ is a line,
is a BN-curve.
It is evidently of degree $d + 1$
and genus $g$,
and its hyperplane section contains the hyperplane section of $f$,
plus the independently general point $\hat{f}(\pp^1) \cap H$,
and thus contains $d + 1 - n$ general points as desired.
\end{proof}

\begin{lm} \label{exc-good}
Suppose that, for some integer $b \geq 0$,
we have $b \leq d - n$, and $b \leq g$, and $\rho(d, g, r) \geq b$, and that
\[2d + (r - 1) n - (r - 3)g - 4b - 2 \geq  0.\]
Then $(d, g, r, n)$ is good if
$(d - b, g - b, r, n)$ is excellent.
\end{lm}
\begin{proof}
If $b = 0$, the result is obvious; we thus suppose $b \geq 1$.
Note that our assumptions imply $g' := g - b \geq 0$, and that $n$
and $d' := d - b$ satisfy $0 \leq n \leq d'$. Moreover,
\[\rho(d', g', r) = (r + 1)(d - 1) - r(g - 1) - r(r + 1)
= (r + 1)d - rg - r(r + 1) - 1
\geq 0.\]
In particular, $(d', g', r, n)$ are nonnegative integers with $0 \leq n \leq d'$
and $\rho(d', g', r) \geq 0$.
By assumption, $(d', g', r, n)$ is excellent.

So let $f \colon C \to \pp^r$ be a general BN-curve
of degree $d'$ and genus $g'$,
whose hyperplane section
contains a set $D$ of $d' - n' = d - n - b$ general points.
By assumption,
\[\chi(N_f(-D)) = (r + 1)(d - b) - (r - 3)(g - b - 1) - (r - 1) (d - b - n) \geq (r - 1)(b + 1),\]
and so $f$ passes through a set $S$ of $b + 1$ points
that are general, independent of $D$.

Let $C'$ be a rational curve of degree $b$ through $S$.
Then $\hat{f} \colon C \cup_S C' \to \pp^r$
is a BN-curve by Theorem~1.6 of~\cite{rbn}.
It is evidently of degree $d' + b = d$ and genus $g' + b = g$,
and its hyperplane section is the union of the hyperplane sections
of $\hat{f}|_C$ and $\hat{f}|_{C'}$, which contain independently general
sets of $d - b - n$ and $b$ points by construction, for $d - n$
general points in total.
\end{proof}

\begin{lm}
Suppose that nonnegative integers $d_1, g_1, d_2, g_2, n_1, n_2, k$
with $\rho(d_i, g_i, r) \geq 0$ and $n_i \leq d_i$ for $i \in \{1, 2\}$, and with $k \geq 1$, satisfy
\begin{align*}
(r + 1) d_1 - r g_1 + r &\geq rk \\
2 d_2 - (r - 3) (g_2 - 1) &\geq (r - 1)(k - n).
\end{align*}
Then $(d_1 + d_2, g_1 + g_2 + k - 1, r, n_1 + n_2)$ is good if $(d_1, g_1, r, n_1)$ is good and $(d_2, g_2, r, n_2)$
is excellent.
\end{lm}
\begin{proof}
Since $(d_1, g_1, r, n_1)$ is good by assumption,
we may let $f_1 \colon C_1 \to \pp^r$ be a general BN-curve
of degree $d_1$ and genus $g_1$,
whose hyperplane section contains a set $D_1$ of
$d_1 - n_1$ general points.
Since $(r + 1) d_1 - r g_1 + r \geq rk$,
Corollary~1.3 of~\cite{tang} implies that
$f_1$ passes through a set $S$ of $k$ general points.

Since by assumption
$(d_2, g_2, r, n_2)$ is excellent and
$2 d_2 - (r - 3) (g_2 - 1) \geq (r - 1)(k - n)$,
we may let $f_2 \colon C_2 \to \pp^r$ be a general BN-curve
of degree $d_2$ and genus $g_2$, passing through $S$
and an independently general set $D_2$ of $d_2 - n_2$ general points.

By Theorem~1.6 of~\cite{rbn}, the curve $C_1 \cup_S C_2 \to \pp^r$
is a BN-curve; by inspection, its hyperplane section contains a set
$D_1 \cup D_2$ of $d_1 + d_2 - n_1 - n_2$ general points.
\end{proof}

The remainder of the paper is a purely combinatorial argument to show that
the above inductive arguments, together with the base cases established by Theorem~\ref{main-1},
suffice to prove Theorem~\ref{main-1s}.

\begin{lm} \label{wh} Suppose that $(d, g, r, n)$ satisfy the hypothesis of Theorem~\ref{main-1s}.
Write
\[a = \left\lceil \frac{r - 2}{2} \right\rceil.\]
Suppose that:
\begin{itemize}
\item Either $g \leq 2r$ with strict inequality for $r \geq 6$,
or $n \leq 2r - 7$;
\item If $r \leq 39$, then
\[g \geq \frac{(5r - 7) n - (2r^2 - 9r + 9) \cdot \left\lfloor\frac{n}{a}\right \rfloor - 4r^2 + 94r - 150}{r - 1}.\]
\end{itemize}
Then $(d, g, r, n)$ is good.
In other words, Theorem~\ref{main-1s} holds subject to the above hypotheses.
\end{lm}
\begin{proof}
We apply Lemma~\ref{addhalf} $x$ times, where
\[x = \min\left(\left\lfloor\frac{g}{a + 1}\right \rfloor, \left\lfloor\frac{n}{a}\right \rfloor\right),\]
followed by Lemma~\ref{addline} $y$ times where
\[y = \min(g - (a + 1)x, n - ax),\]
followed by Lemma~\ref{exc-good} with $b = z$ where
\[z = \min(g - (a + 1)x - y, d - n, 10).\] % ,6) ?
This can be done so long as
\begin{align*}
\rho(d - ax, g - (a + 1)x, r) &\geq y + z \\
2(d - ax - y) + (r - 1) (n - ax - y) - (r - 3)(g - (a + 1)x - y) - 4z - 2 &\geq 0;
\end{align*}
if these inequalities hold, then we are reduced to showing that
\[(d', g', r, n') := (d - ax - y - z, g - (a + 1)x - y - z, r, n - ax - y)\]
is excellent.

By definition of $y$, either $n' = n - ax - y = 0$,
or $g - (a + 1)x - y = 0$; in the second case, by definition of $z$,
we also have $z = 0$ and so $g' = g - (a + 1)x - y - z = 0$.

Next, by definition of $z$, either $z = d - n$ (which gives $d' = n'$), or $z = 10$, or we have
$z = g - (a + 1)x - y$ (which gives $g' = 0$).
Since $(d', g', r, n')$ is automatically excellent when $g' = 0$ by
Proposition~\ref{base:excellent},
and we cannot have $d' = n'$ if $n' = 0$,
the only case that
remains to consider is when $n' = 0$ and $z = 10$ but $g' \neq 0$.
Note that this forces $x = \lfloor n / a \rfloor$ and $g' = g - n - x - 10$;
in particular, $g \geq n + 11$, which since either $g \leq 2r$ or $n \leq 2r - 7$, forces $n \leq 2r - 7$.

In this case, we can invoke Theorem~\ref{main-1} to conclude that
$(d', g', r, n')$ is excellent, as desired, provided that
\[(2r - 3) (d - ax - y - z) - (r - 2)^2 (g - (a + 1)x - y - z) - 2r^2 + 3r - 9 \geq 0.\]
Combining all these inequalities and substituting $y = n - ax$ and $z = 10$, all that must be done
to complete the proof is to show:
\begin{align*}
rx + (r + 1) d - rg - n - r^2 - r - 10 &\geq 0 \\
(r - 3) x + 2 d - (r - 3) g + (r - 5) n - 42 &\geq 0 \\
(2r - 3) d + (r - 2)^2 (x - g) + (r^2 - 6 r + 7) n + 8r^2 - 57r + 61 &\geq 0.
\end{align*}
Using the hypothesis of Theorem~\ref{main-1s} to bound $d$ from below,
and recalling that $g' = g - n - x - 10$, these inequalities follow from
\begin{align}
\frac{r^3 - 5r^2 + 3r + 4}{r + 1} g' + 10r^2 - 62r + 82 &\geq (2r - 3)n - (r - 2)^2 x \label{df} \\
\frac{r - 1}{2} g' + 2 r^2 - 42 r + 70 &\geq (2r - 3)n - (r - 2)^2 x \label{ds} \\
10 r^2 - 62 r + 73 &\geq (2r - 3) n - (r - 2)^2 x. \label{dt}
\end{align}
Note that $(r - 2)^2 / (2r - 3) \leq (r - 2)/2 \leq a$,
and that
$n \leq 2r - 7 < 4 \cdot (r - 2)/2 \leq 4a$.
Since $x = \lfloor n/a \rfloor$, this implies
\[(2r - 3) n - (r - 2)^2 x \leq (2r - 3) \cdot (4a - 1) - (r - 2)^2 \cdot 3 \leq (2r - 3) \cdot (2r - 3) - (r - 2)^2 \cdot 3 = r^2 - 3.\]
Since $g' \geq 1$, this implies \eqref{df} for $r \geq 5$, and \eqref{ds}
for $r \geq 40$, and \eqref{dt} for $r \geq 6$.
When $r = 5$, we have $a = 2$ and $n \leq 3$;
for each of these four values of $n$, we easily verify \eqref{dt}.
It thus remains to verify \eqref{ds} for 
$r \leq 39$. But upon rearrangement (using $g' = g - n - x - 10$ and $n = \lfloor n/a \rfloor$),
this is exactly our assumption that
\[g \geq \frac{(5r - 7) n - (2r^2 - 9r + 9) \cdot \left\lfloor \frac{n}{a}\right\rfloor - 4r^2 + 94r - 150}{r - 1}. \qedhere\]
\end{proof}

\begin{cor} \label{done} Suppose that $(d, g, r, n)$ satisfy the hypothesis of Theorem~\ref{main-1s}.
If either $g \leq 2r$ with strict inequality for $r \geq 6$, or $n \leq 2r - 7$,
then $(d, g, r, n)$ is good. In other words, Theorem~\ref{main-1s} holds subject to the hypothesis
that either $g \leq 2r$ with strict inequality for $r \geq 6$, or $n \leq 2r - 7$.
\end{cor}
\begin{proof}
From Lemma~\ref{wh}, it remains to check the finitely
many values of $(g, r, n)$ with $5 \leq r \leq 39$ and $n \leq 2r - 7$ and 
\[g \leq \frac{(5r - 7) n - (2r^2 - 9r + 9) \cdot \left\lfloor \frac{n}{a}\right\rfloor - 4r^2 + 94r - 151}{r - 1} \twhere a = \left\lceil \frac{r - 2}{2} \right \rceil.\]
And for each such triple, by Lemma~\ref{addline-1}, we just have to
check the minimal value of $d$ satisfying the hypotheses of Theorem~\ref{main-1s};
all that remains is thus a finite computation, 
which is done in Appendix~\ref{appendix:code}.
\end{proof}

\begin{proof}[Proof of Theorem~\ref{main-1s}]
We argue by induction on $d$, with bases cases given by Corollary~\ref{done}.
For the inductive step, we may suppose $g \geq 2r$ with strict inequality for $r = 5$,
and $n \geq 2r - 6$. In particular, we may apply Lemma~\ref{addcan},
which asserts that $(d, g, r, n)$ is good provided that $(d - 2r + 2, g', r, n - 2r + 6)$ is good.
Note that
\begin{multline*}
(2r - 3)((d - 2r + 2) + 1) - (r - 2)^2 (g' - (n - 2r + 6)) - 2r^2 + 3r - 9 \\
= (2r - 3)(d + 1) - (r - 2)^2 (g - n) - 2r^2 + 3r - 9 + \begin{cases}
2r^2 - 14r + 18 & \text{if $r \geq 6$;} \\
7 & \text{if $r = 5$.}
\end{cases}
\end{multline*}
Since $(d, g, r, n)$ satisfies the assumptions of Theorem~\ref{main-1s},
so does
$(d - 2r + 2, g', r, n - 2r + 6)$; thus
$(d - 2r + 2, g', r, n - 2r + 6)$ 
is good by our inductive hypothesis as desired.
\end{proof}

\appendix 

\section{Code for Theorem~\ref{main-1s} \label{appendix:code}}

In this section, we give python code to do the finite
computations described in Section~\ref{ghs};
running this code produces no output, thus verifying Theorem~\ref{main-1s}
in these cases.

\begin{lstlisting}
def excellent(d, g, r, n):
	if (2 * r - 3) * d - (r - 2)**2 * g - 2 * r**2 + 3 * r - 9 >= 0:
		return True
	
	return ((d >= g + r) and ((d, g, r) not in ((5, 2, 3), (6, 2, 4), (7, 2, 5))) and (2 * d - (r - 3) * (g - 1) >= (r - 1) * (g - n)))
		
GOOD = {}
def good(d, g, r, n):
	dgrn = (d, g, r, n)

	if GOOD.has_key(dgrn):
		return GOOD[dgrn]
		
	if excellent(d, g, r, n):
		GOOD[dgrn] = True
		return True
	
	a = (r - 1) / 2

	if (g >= a + 1) and (n >= a):
		if good(d - a, g - a - 1, r, n - a):
			GOOD[dgrn] = True
			return True
	
	if (g >= 1) and (n >= 1) and ((r + 1) * d - r * g - r * (r + 1) >= 1):
		if good(d - 1, g - 1, r, n - 1):
			GOOD[dgrn] = True
			return True

	for i in xrange(1, min(d - n, g, (r + 1) * d - r * g - r * (r + 1)) + 1):
		if 2 * d + (r - 1) * n - (r - 3) * g - 4 * i - 2 >= 0:
			if excellent(d - i, g - i, r, n):
				GOOD[dgrn] = True
				return True

	for g1 in xrange(g + 1):
		for g2 in xrange(g + 1 - g1):
			k = g + 1 - g1 - g2
			
			for n1 in xrange(n):
				n2 = n - n1
	
				d1min = max(n1, r + (r * g1 + r) / (r + 1))
				d2min = max(n2, r + (r * g2 + r) / (r + 1))
	
				for d1 in xrange(d1min, d + 1 - d2min):
					d2 = d - d1
	
					if (r + 1) * d1 - r * g1 + r >= r * k:
						if 2 * d2 - (r - 3) * (g2 - 1) >= (r - 1) * (k - n):
							if good(d1, g1, r, n1) and excellent(d2, g2, r, n2):
								GOOD[dgrn] = True
								return True

	GOOD[dgrn] = False
	return False

for r in xrange(5, 40):
	a = (r - 1) / 2
	nmax = 2 * r - 7
	for n in xrange(nmax + 1):
		x = n / a
		gmax = ((5 * r - 7) * n - (2 * r**2 - 9 * r + 9) * x - 4 * r**2 + 94 * r - 151) / (r - 1)
		for g in xrange(gmax + 1):
			dmin1 = r + (r * g + r) / (r + 1)
			dmin2 = ((r - 2)**2 * (g - n) + 2 * r**2 - 3 * r + 9 + (2 * r - 4)) / (2*r - 3) - 1
			d = max(dmin1, dmin2, n)

			if not good(d, g, r, n):
				print d, g, r, n
\end{lstlisting}

\bibliography{ibe}{}
\bibliographystyle{plain}

\end{document}